\documentclass[11pt]{amsart}
\usepackage{amsmath}
\usepackage{amsfonts}
\usepackage{float}
\usepackage{mathrsfs}
\usepackage{mathtools}
\usepackage{latexsym}
\usepackage{amsmath,amsthm,amssymb, amscd,color}
\usepackage{mathrsfs}
\usepackage[all]{xy}
\usepackage[all]{xy}
\usepackage{graphicx}
\usepackage{resizegather}
\usepackage{tikz}
\usepackage{tikz-cd}

\usepackage{blkarray} 
\usepackage{framed}

\usepackage{hyperref}
\hypersetup{
    colorlinks = true,
    linkcolor = {blue},
    linkbordercolor = {white},
    citecolor = {magenta},
    urlcolor = {cyan},
   }

\theoremstyle{plain}
\newtheorem{theorem}{Theorem}[section]

\numberwithin{equation}{section}

\theoremstyle{definition}
\newtheorem{definition}[theorem]{Definition}
\newtheorem{example}[theorem]{Example}

\newtheorem{proposition}[theorem]{Proposition}
\newtheorem{remark}[theorem]{Remark}

\theoremstyle{remark}

\newcommand{\RR}{\mathbb{R}}

\begin{document}

\title[Invariant algebraic sets of certain vector fields]{Invariant Hyperplane sections of vector fields on the product of spheres}

\author[Joji B.]{Joji Benny}
\address{Department of Mathematics, Indian Institute of Technology Madras, India}
\email{jojikbenny@gmail.com}

\author[S. Sarkar]{Soumen Sarkar}
\address{Department of Mathematics, Indian Institute of Technology Madras, India}
\email{soumen@iitm.ac.in}

\date{\today}
\thanks{}

\abstract 
Let $S_{p,q}$ be the hypersurface in $\mathbb{R}^{n},$ where $n=p+q+1,$ defined by the following:
\begin{gather*}
 S_{p,q} := \left\lbrace (x_1,\ldots,x_{n}) \in \mathbb{R}^{n} ~~ \big| ~~ \left( \sum_{i=1}^{p+1} x_i^2 - a^2 \right)^2 + \sum_{j=p+2}^n x_j^2 = 1 \right\rbrace \end{gather*}
where $a > 1$. We show that $S_{p,q}$ is homeomorphic to the product $S^p \times S^q$. We classify all degree one and two polynomial vector fields on $S_{p,q}$. We consider the polynomial vector field $\mathcal{X} = (R_1,...,R_{p+1},R_{p+2},...,R_n)$ in $\mathbb{R}^{p+q+1}$ which keeps $S_{p,q}$ invariant. Then we study the number of certain invariant algebraic subsets of $S_{p,q}$ for the vector field $\mathcal{X}$ if either $p>1$ or $q>1$.
\endabstract

\maketitle

\section{Introduction}
Let  $R_1,\ldots,R_n$ be polynomials in $\mathbb{R}[x_1,\ldots,x_n]$. Then the following system of differential equations
\begin{equation} \label{eq: I1}
 \frac{dx_i}{dt} = R_i(x_1, \ldots,x_n)       \end{equation}
for $i = 1,\ldots,n$ is called a polynomial differential system in $\mathbb{R}^n$. The differential operator
\begin{equation}  \label{eq: I2}
 \mathcal{X} = \sum_{i=1}^n R_i \frac{\partial}{\partial x_i}    
\end{equation}
is called the vector field associated with the system \eqref{eq: I1}. The degree of the polynomial vector field in \eqref{eq: I2} is defined to be $\max \{ \deg(R_i) | i=1,\ldots,n \}$. 

 When $n=2$ in \eqref{eq: I1}, this differential system has been studied since 1900 possibly because of the second part of the Hilbert sixteenth problem (see \cite{Ila02} and some references therein). 
 
 An {\it invariant algebraic set} for \eqref{eq: I2} is a subset $A \subset \mathbb{R}^n$ such that $A$ is the zero set of some $f(x_1,x_2,\ldots,x_n) \in \mathbb{R}[x_1,x_2, \ldots,x_n]$ and $\mathcal{X}f = Kf$ for some $K \in \mathbb{R}[x_1,x_2, \ldots,x_n]$. Here the polynomial $K$ is called the \textit{cofactor} of $f$.

In the case of the torus $S^1 \times S^1$, the maximum number of invariant \textit{meridians} and  \textit{parallels} are studied in \cite{LliMed11} and \cite{LliReb13}. Bounds on the number of invariant hyperplanes and co-dimension one spheres for polynomial vector fields in $\mathbb{R}^n$ are obtained in \cite{LliMed07} and \cite{BolLli11} respectively. In relation to the second part of  Hilbert's sixteenth problem, the maximum number of {\it algebraic limit cycles} of a polynomial vector field in $\mathbb{R}^2$ as a function of its degree has been studied in \cite{LRS10}, \cite{LRS10bis}, and \cite{Zh11}. The maximum number of straight lines that are invariant for a vector field in the real plane as a function of its degree has been studied in \cite{AGL98} and \cite{So96}. 

Inspired by the above works, we introduce an algebraic representation of $S^p \times S^q$ for any positive integers $p, q$ and study the number of invariant algebraic sets of a vector field on $S^p \times S^q$. We are primarily interested in the algebraic sets obtained by the intersection of $S^p \times S^q$ with hyperplanes.  We classify all degree one and two polynomial vector fields on $S_{p,q}$. We  obtain an upper bound of the number of possible invariant algebraic sets that are intersections of $S^p \times S^q$ with hyperplanes. Therefore, many other types of invariant hypersurfaces in $S^p \times S^q$ remain to be explored.  

The paper is organized as follows. In Section \ref{sec:basic_on_vf}, we recall the definition of the \textit{extactic algebraic polynomial} associated to a vector subspace of the ring of polynomials and the given vector field. We also state some basic properties of the extactic polynomials.

In Section \ref{sec:vf_on_spsq}, we show that $S_{p,q}$ with the subspace topology is homeomorphic to $S^p \times S^q$. Then we classify all degree one and two polynomial vector fields on $S_{p,q}$.

In Section \ref{sec:inv_hyperplanes}, we define {\it meridians} and  {\it parallels} on $S_{p,q}$ analogously to their definitions given in \cite{LliMed11} for $S^1 \times S^1$. We prove that meridians and parallels are connected if $p >1$ and $q>1$. We give an upper bound for the number of invariant meridians and parallels in Theorem \ref{thm:5}. We compute an upper bound for the number of invariant meridians of degree one vector fields on $S_{p,q}$. The bound for the degree one vector fields on $S_{p,q}$ is attained.
We demonstrate that the bounds in Theorem \ref{thm:5} are close to being \textit{tight} for the cases $p=2,3$ and $\deg (\mathcal{X}) \geq 4$. 

In Section  \ref{sec:inv_planes_s1sq}, we show that the maximum number of invariant {\it meridians} on $S_{1,q} (\cong S^1 \times S^q )$ is $2(m-1)$ where $m$ is the degree of $\mathcal{X}$ on $S_{1,q}$ and that this bound can be reached. We also discuss invariant {\it parallels} on $S_{1,q}$ and derive a bound on the number of invariant parallels. We show that there exists a vector field on $S_{1,q}$ with rational first integral.

In Section \ref{sec:inv_planes_sps1}, we show that the maximum number of invariant parallels for $\mathcal{X}$ on $S_{p,1} := S^p \times S^1$ is $2(m-2)$ where $m= \deg \mathcal{X}$. We also demonstrate that this bound can be attained. We also make a remark on an upper bound for the number of invariant meridians for $\mathcal{X}$ on $S_{p,1}$ and exhibit a vector field on $S_{p,1}$ with rational first integral.

\section{Invariant Algebraic Sets and Extactic polynomials}\label{sec:basic_on_vf}
In this section, we recall the concept of invariant algebraic sets and extactic polynomials for polynomial vector fields on $\mathbb{R}^n$ following \cite{LliMed11}. Then we discuss some basic properties of extactic polynomials.

Let $S$ be a hypersurface in $\mathbb{R}^n$  defined by the zeroes of a non constant polynomial $h \in \mathbb{R}[x_1,x_2, \ldots,x_n]$. We say that a vector field $\mathcal{X}$ of the form \eqref{eq: I2} is defined on $S$ if $(R_1,R_2,\ldots,R_n)\cdot \nabla h=0$ for all points on the hypersurface $S$. This is equivalent to saying that $\mathcal{X}h= Kh$ for some polynomial $K \in \mathbb{R}[x_1,x_2, \ldots,x_n]$.  Because  $h(a_1,a_2,...,a_n)=0$ for all points $(a_1,a_2,...,a_n)$ on the hypersurface S. The hypersurface  $S$ is called a regular hypersurface if   $\nabla h \neq 0$ for all points on $S$. This hypersurface is called irreducible if $h$ is irreducible. The \textit{degree} of an irreducible hypersurface $S$ is defined to be the degree of $h$.

In order to study invariant algebraic sets on an algebraic hypersurface $S \subset \mathbb{R}^n$,  one may  use  the idea of {\it extactic algebraic polynomial}. We briefly recall this concept following \cite{LliMed07}.
Let $W$ be a $k$-dimensional vector subspace of $\mathbb{R}[x_1,x_2,x_3,\dots,x_n]$ with basis $\{     v_1,\dots,v_k\}$. Then the extactic algebraic polynomial of the vector field $\mathcal{X}$ associated to $W$ is given by 
\begin{equation*}
\mathcal{E}_W(\mathcal{X}) = \mathsf{det} \begin{pmatrix}
v_1 &  v_2 & \cdots &   v_k \\ \mathcal{X}(v_1) & \mathcal{X}(v_2) & \cdots & \mathcal{X}(v_k) \\ \vdots & \vdots & \ddots  & \vdots \\ \mathcal{X}^{k-1}(v_1) & \mathcal{X}^{k-1}(v_2) & \cdots & \mathcal{X}^{k-1}(v_k) \end{pmatrix},
\end{equation*}
 where $\mathcal{X}^j(v_i) = \mathcal{X}^{j-1}(\mathcal{X}(v_i))$ for $i \geq 2$. We note that the definition of the extactic algebraic set is independent of the choice of basis of $W$, see  Section 2 of \cite{LliMed11}. In this paper we will work with $W$ mostly of the form $\{x_1,x_2, \ldots, x_n   \}$ or $\{1, x_1,x_2, \ldots, x_n      \}.$
  
   We recall the definition of the {\it algebraic multiplicity} of an irreducible algebraic set given by an irreducible polynomial $f=0$ from \cite{LliZh09}.

\begin{definition}  \label{def: 21}
The hypersurface given by $f=0$ with $f \in W$ has algebraic multiplicity, or simply, {\it multiplicity} $m$ for $\mathcal{X}$ if $\mathcal{E}_W(\mathcal{X}) \neq 0$ and $(f)^m$ divides   $\mathcal{E}_W(\mathcal{X}) $ and for $m' > m$, $(f)^{m'}$ is not a factor of $\mathcal{E}_W(\mathcal{X}) $. It has no defined algebraic multiplicity if $\mathcal{E}_W(\mathcal{X}) = 0$. 
 \end{definition}


We note that, in this paper, we are interested in hypersurfaces with finite multiplicity which is accounted when we count the number of invariant algebraic hypersurfaces. 
  
  We shall use the following proposition whose proof can be found in \cite[Proposition 1]{LliMed07}.
 
 \begin{proposition} \label{Prop:1}
  Let $\mathcal{X}$ be a polynomial vector field on $\mathbb{R}^n$ and $W$ a finite dimensional vector sub-space of $\mathbb{R}[x_1,x_2, \dots,x_n]$ with $\dim(W) >1$. If $\{ f=0  \}$ is an invariant algebraic set for the vector field $\mathcal{X}$ and $f \in W,$ then $f$ is a factor of $\mathcal{E}_W (\mathcal{X}).$
  \end{proposition}

We recall that a function $g$ is called a first integral of the system  \eqref{eq: I2} if $\mathcal{X}g=0$. If $g$ is a rational function then  $g$ is called a \textit{rational first integral}. If the differential system \eqref{eq: I1} has a first integral, then the system possesses infinitely many invariant algebraic sets. A proof of this fact can be found on page 102 of \cite{Jou79}. We quote the following result from \cite{LliBol11}.

\begin{proposition} \label{thm:1}
Let $S$ be a regular algebraic hypersurface of degree $d$ in  $\mathbb{R}^{n+1}$. The polynomial vector field $\mathcal{X}$ on $S$ of degree $m>0$ admits ${n+m \choose n+1} - {n+m-d \choose n+1} +n$ invariant algebraic hypersurfaces irreducible in $\mathbb{C}[x_1,x_2,\cdots,x_{n+1}]$ if and only if $\mathcal{X}$ has a rational first integral.
\end{proposition}

\section{Degree one and degree two vector fields on \texorpdfstring{$S^p \times S^q$}{ product of spheres } }\label{sec:vf_on_spsq}

In this section, we show that $S^p \times S^q$ is homeomorphic to the zero set of a polynomial. Then we characterize degree one and degree two polynomial vector fields on it. For the remainder of the paper, we fix $n = p+q+1.$

\subsection{\texorpdfstring{$S^p \times S^q$}{Product of spheres}  as a hypersurface and its meridians and parallels}
 We consider the following hypersurface in $\mathbb{R}^{p+q+1}$ defined by the polynomial identity
\begin{equation}    \label{eq:30}
(x_1^2 + x_2^2 + \cdots + x_{p+1}^2 - a^2)^2 + x_{p+2}^2 +x_{p+3}^2+ \cdots + x_n^2 = 1,
\end{equation} 
where $a > 1$ . We denote this hypersurface by $S_{p,q}$.    We show that $S_{p,q}$ is homeomorphic to the product $S^p \times S^q$ of two spheres. Then we determine when the vector field $\mathcal{X}$ in \eqref{eq: I2} is defined on $S_{p,q}.$

Let $0 \leq c \leq 1$ and $x_{p+2}^2 +x_{p+3}^2+ \cdots + x_n^2 = c$. Then  $$x_1^2 +x_2^2 + \cdots + x_{p+1}^2 = \pm \sqrt{1-c} + a^2.$$
Let
 $$ {\scriptstyle U_1 ~~~ := ~~~ \left\lbrace (x_1,\ldots,x_{p+1},x_{p+2},\ldots,x_n) \in \mathbb{R}^n \big|  \sum_{i=1}^{p+1} x_i^2  = a^2 + \sqrt{1-c}, \quad \sum_{j=p+2}^n x_j^2 =c, \quad 0 \leq c \leq 1 \right\rbrace    } $$ and 
$${\scriptstyle U_2 ~~~ := ~~~ \left\lbrace (x_1,\ldots,x_{p+1},x_{p+2},\ldots,x_n) \in \mathbb{R}^n \big|  \sum_{i=1}^{p+1} x_i^2  = a^2 - \sqrt{1-c}, \quad \sum_{j=p+2}^n x_j^2 =c, \quad 0 \leq c \leq 1 \right\rbrace  . }$$
Observe that $U_1$ and $U_2$ are homeomorphic to $S^{p} \times D^q$. They are identified along their boundary $S^p \times S^{q-1}$ via the identity map, where 
\begin{gather*}
    S^{p} \times S^{q-1} = \left\lbrace (x_1, \ldots, x_{p+1}, x_{p+2}, \ldots, x_n) \in \mathbb{R}^n \big| \sum_{i=1}^{p+1} x_i^2 = a^2, \quad \sum_{j=p+2}^n x_j^2 = 1   \right\rbrace .\end{gather*}
Thus
 $$S_{p,q} = ~~~ U_1 \bigcup_{S^{p} \times S^{q-1}} U_2 ~~~ = ~~~ S^{p} \times D^q \bigcup_{S^{p} \times S^{q-1}} S^{p} \times D^q \cong ~~~ S^{p} \times S^q.$$

 Note that the product $S^p \times S^{q-1}$ in the above equation is a subset of $\mathbb{R}^n$, whereas the representation of $S^p \times S^{q-1}$ as in \eqref{eq:30} says $S^p \times S^{q-1}$ is a subset of $\mathbb{R}^{n-1}$. These two representations of $S^p \times S^{q-1}$ are different in two different ambient spaces.

Writing out the condition for $\mathcal{X}$ to be invariant on the hypersurface $S_{p,q}$, we get the following.
\begin{equation} \label{eqn: 33}
4\left( \sum_{i=1}^{p+1}x_i^2  -a^2 \right) \left( \sum_{i=1}^{p+1}x_iR_i \right) + 2 \left( \sum_{j=p+2}^n x_j R_j \right)  =  K \left( (\sum_{i=1}^{p+1}x_i^2 - a^2)^2 + \sum_{j=p+2}^nx_j^2 - 1 \right),
\end{equation}
for some $K \in \mathbb{R}[x_1, \ldots, x_n]$.
One defines the \textit{degree vector} of the vector field $\mathcal{X}$ to be $\overline{m}:=(m_1,m_2,\ldots,m_n)$ where $m_i = \deg(R_i)$ for $i=1,2,\ldots,n$.

\begin{definition}
\begin{enumerate}
\item The intersection of $S_{p,q}$ and the hyperplane \\ $ \{ \sum_{i=1}^{p+1} a_ix_i = 0 \}$ where $a_i \in \mathbb{R}$, for $i=1,\ldots,p+1$ is called a `meridian' on $S_{p,q}$.

\item The intersection of $S_{p,q}$ and the hyperplane $ \{\sum_{j=p+2}^n a_jx_j - c = 0\}$ for some $c \in (0,1)$ where $a_j \in \mathbb{R}$, for $j=p+2,\ldots, n$ is called a `parallel'  on $S_{p,q}$.
\end{enumerate}
\end{definition}

We say that a meridian on $S_{p,q}$ is invariant by the flow of the polynomial vector field $\mathcal{X}$ on $S_{p,q}$ if $\mathcal{X}( \sum_i^{p+1} a_ix_i) = K(\sum_i^{p+1} a_ix_i)$ where $a_i \in \mathbb{R}$ for  $ i=1,\ldots,p+1$ and for some $K \in \mathbb{R}[x_1,\ldots,x_{p+1},x_{p+2},\ldots,x_n]$.

Similarly, one can define an invariant parallel on $S_{p,q}$. We note that these definitions generalize the definitions of the invariant meridians and the invariant parallels on $ S^1 \times S^1$ of \cite{LliMed11}. 


\subsection{Degree one vector fields on \texorpdfstring{$S^p \times S^q$}{ product of spheres } }
Let
\begin{equation} \label{eq: a21}
R_i = \sum_{j=1}^{n} c_{ij}x_j + c_{i0} ,
 \end{equation}
where $c_{ij} \in \mathbb{R}$ for all $j$ and $i=1,\ldots,n$. 
\begin{proposition}      \label{prop: bx32}
Let $\mathcal{X} =  \sum_i R_i \frac{\partial}{\partial x_i} $ be a degree one vector field defined on   $S_{p,q}$ where $R_i$s are given by \eqref{eq: a21}. Then the matrix $(c_{ij})_{p+1 \times p+1}$  is skew symmetric and  $c_{i0} =0$ for all $i \in \{1, \ldots, n\}$.
\end{proposition}
\begin{proof}
For this case, we see that  $K$ in \eqref{eqn: 33} is a constant since the degree of the expression on the left is four and the expression in brackets on the right already has degree four.

In fact, $K$ vanishes since there is a constant term on the right of \eqref{eqn: 33}, but all terms on the left have degree at least one.
So for linear vector fields, \eqref{eqn: 33} becomes
 \begin{equation} \label{eqn: a33}
4\left( \sum_{i=1}^{p+1}x_i^2  -a^2 \right) \left( \sum_{i=1}^{p+1}x_iR_i \right) + 2 \left( \sum_{j=p+2}^n x_j R_j \right)  = 0.
\end{equation}

Observe that in the expression for $R_i$, $i = 1,\ldots,p+1$, there can be no $x_k$s for $k> p+1$. Since if there was, then \eqref{eqn: a33} will have terms of the form $4x_i^2x_jx_k$, $k> p+1$ which cannot be canceled out. Hence this will violate \eqref{eqn: a33}. Similarly, $R_j$ cannot have any terms with $x_i$s, where $1 \leq i \leq p+1$  (for $j = p+2,\ldots,n$).
 By similar reasoning we get that $c_{i0} = d_{j0}=0$ for all $i$ and $j$.

Substituting $R_i$ from \eqref{eq: a21} in \eqref{eqn: a33} and collecting coefficients of $x_{i_{0}}$ for a fixed $i_0$, we have
\begin{equation} \label{eq: a23}
4(\sum_{i=1}^{p+1}x_i^2 -a^2) \{R_{i_0}  + \sum_{s=1}^{p+1}c_{si_0}x_s \} =0.
\end{equation}
Since $R_{i_0} = \sum_{j=1}^{p+1}c_{i_0j}x_j$, then \eqref{eq: a23} becomes 
 \begin{equation*} \label{eq: a3}
4\left( \sum_{i=1}^{p+1}x_i^2 -a^2 \right) \left\lbrace \sum_{j=1}^{p+1}c_{i_0j}x_j + \sum_{s=1}^{p+1}c_{si_0}x_s \right\rbrace =0.
\end{equation*}
Re-indexing and gathering coefficients of $x_j$, we get  
$$(c_{i_0j} + c_{ji_0}) =0,$$
since there is a dense open subset of $S_{p,q}$ with $\sum_{i=1}^{p+1}x_i^2 \neq a^2$.
From this, we see that the real matrix determined by $\{c_{ij} \}$ is skew-symmetric.
\end{proof}

\begin{remark}
We see that every first degree vector field defined on $S_{p,q}$ has a first integral by \eqref{eqn: a33}.
\end{remark}

\subsection{Degree two vector fields on \texorpdfstring{$S^p \times S^q$}{product of spheres}}   \label{sec: 33}

 Assume that $\mathcal{X} = \sum R_i \frac{\partial}{\partial x_i}$ is a degree two vector field. Then we have
\begin{equation}  \label{eq: b6}
    R_i = \sum_{j \leq k} \beta_{ijk} x_jx_k + \sum_j \beta_{ij}x_j + \beta_{i0},  \end{equation}
where $\beta_{ijk}, \beta_{ij}, \beta_{i0}$ belong to $\mathbb{R}$ for $i, j, k \in \{ 1,2,\ldots,n \}$ with $j \leq k$. Since this vector field $\mathcal{X}$ satisfies \eqref{eqn: 33}, the polynomial $K$ has to be linear, say $K = \alpha_1x_1 + \alpha_2 x_2 + \cdots + \alpha_n x_n + \alpha_0$ for some $\alpha_0, \ldots, \alpha_n \in \mathbf{R}$. We get $\alpha_0 = 0$, as there is no constant term on the left of \eqref{eqn: 33}. 

 Since $R_i$s in  \eqref{eq: b6}  satisfy \eqref{eqn: 33}, equating the degree five terms and cancelling $\sum_{i=1}^{p+1} x_i^2$ we have

\begin{equation}   \label{eq: b7}
\begin{aligned}
4 \left( x_1 \sum_{j \leq k} \beta_{1jk}x_jx_k + x_2 \sum_{j \leq k} \beta_{2jk}x_jx_k + \cdots + x_{p+1} \sum_{j \leq k} \beta_{p+1jk}x_jx_k \right) = \\ (\alpha_1 x_1 + \alpha_2 x_2 + \cdots + \alpha_n x_n) (x_1^2 + x_2^2 + \cdots +x_{p+1}^2).
\end{aligned}
\end{equation}
Observe that in \eqref{eq: b7} there are no terms of the form $x_ix_jx_k$ for $i,j,k$ pairwise distinct on the right hand side, therefore such terms must sum to zero on the left hand side. This gives
$$ 4(\beta_{ijk} + \beta_{jik} + \beta_{kij})  = 0,$$
when $i,j,k \in \{ 1,\ldots,p+1\}.$ 
If one of $i,j,k$ is in $\{p+2,\ldots,n  \}$, say $k$, then
$$ 4(\beta_{ijk} + \beta_{jik}) = 0.  $$
For the remaining terms, we have,
\begin{equation*}
 4 \beta_{iii} = \alpha_i ~~~~\mbox{and}~~~~ 4( \beta_{ijj} + \beta_{jij} ) = \alpha_i
\end{equation*}
for $i,j \in \{ 1,\ldots,p+1  \},$  and assuming without loss of generality that $i<j$. We also have $4(\beta_{iik}) = \alpha_k,$ for $i \in \{1,\ldots,p+1\}$ and $k \in \{p+2,\ldots,n\}.$

Observe that there cannot be terms of the form $x_ix_jx_k$ with any two of $i,j,k$ belonging to $\{p+2,\ldots,n\}$.

 Notice that there are no degree four terms on the right hand side, this implies
$$4(x_1^2 + \cdots + x_{p+1}^2)(x_1 \sum_j \beta_{1j}x_j + \cdots + x_{p+1}\sum_j \beta_{p+1}x_j) =0.   $$ 
This gives 
\begin{equation} \label{eq: b8}
\beta_{ij} + \beta_{ji} = 0
\end{equation}
for $i,j \in \{1,\ldots,p+1\}.$

 Next, equating degree three terms gives
\begin{equation*}
\begin{aligned}
4(x_1^2 + \cdots + x_{p+1}^2)(x_1 \beta_{10} + \cdots + x_{p+1}\beta_{p+10}) -\\ 4a^2 \left(   x_1 \sum_{j \leq k} \beta_{1jk}x_jx_k + x_2 \sum_{j \leq k} \beta_{2jk}x_jx_k + \cdots + x_{p+1} \sum_{j \leq k} \beta_{p+1jk}x_jx_k \right) + \\ 2x_{p+2} \sum_{j \leq k} \beta_{p+2jk} x_jx_k + \cdots + 2x_{n} \sum_{j \leq k} \beta_{njk} x_jx_k = \\ (\alpha_1x_1 + \cdots + \alpha_n x_n)\left(-2a^2(x_1^2+ \cdots + x_{p+1}^2) + x_{p+2}^2 + \cdots + x_n^2      \right).
\end{aligned}
\end{equation*} 
  We can rewrite this as, using \eqref{eq: b7}, 
\begin{equation*}
\begin{aligned}
4(x_1^2 + \cdots + x_{p+1}^2)(x_1 \beta_{10} + \cdots + x_{p+1}\beta_{p+10})  + \\ 2x_{p+2} \sum_{j \leq k} \beta_{p+2jk} x_jx_k + \cdots + 2x_{n} \sum_{j \leq k} \beta_{njk} x_jx_k = \\ (\alpha_1x_1 + \cdots + \alpha_n x_n)\left(-a^2(x_1^2+ \cdots + x_{p+1}^2) + x_{p+2}^2 + \cdots + x_n^2      \right).
\end{aligned}
\end{equation*}
In this equation, we see that 
$$ 2 \beta_{kii} = -a^2 \alpha_k,   $$
for $k \in \{ p+2,\ldots,n\}$ and $i \in \{1,\ldots,p+1 \}$. Now we have

\begin{equation}  \label{eq: b9}
4 \beta_{i0} = -a^2 \alpha_i
\end{equation}
for $i \in \{1,\ldots,p+1\}$. If $i,j,k$ are distinct and all belong to $\{p+2,\ldots,n\}$, then
$$ \beta_{ijk} + \beta_{jik} + \beta_{kij} = 0.     $$
If one of $i,j,k$, say $j$, belongs to $\{1,\ldots,p+1\}$, then
$$\beta_{ijk} + \beta_{kji} = 0.    $$
In the remaining cases, we have

\begin{equation} \label{eq: ba1}
 2(\beta_{ijj} + \beta_{jij}) = \alpha_i, ~~~~ \mbox{and} ~~~~  2 \beta_{iii} = \alpha_i.
\end{equation}

 Now equating degree two terms,
\begin{multline*}
-4a^2 \left( x_1 \sum_j \beta_{ij} + \cdots + x_{p+1} \sum_j \beta_{p+1j}x_j\right) + \\ 2x_{p+2}\sum_j \beta_{p+2j} x_j + \cdots + 2x_n \sum_j \beta_{nj} x_j = 0. \end{multline*}
Then using \eqref{eq: b8}, the above equation becomes
$$  2x_{p+2}\sum_j \beta_{p+2j} x_j + \cdots + 2x_n \sum_j \beta_{nj} x_j = 0.     $$
This implies that $\beta_{kl} + \beta_{lk} = 0$ when $k,l \in \{ p+2,\ldots,n\}.$

 Finally, equating degree one terms, we have
\begin{multline*}
-4a^2(x_1 \beta_{10} + \cdots + x_{p+1} \beta_{p+10}) + 2x_{p+2}\beta_{p+20} + \cdots + 2x_n\beta_{n0} = \\ (a^4 -1)(\alpha_1x_1 + \cdots + \alpha_nx_n).
\end{multline*} 
This gives

\begin{equation*}
 \beta_{i0} = \frac{(1-a^4)}{4a^2} \alpha_i,  ~~\mbox{and}~~
 \beta_{j0} = \frac{(a^4 - 1)}{2} \alpha_j,
\end{equation*}
when $i \in \{q,\ldots,p+1\}$ and $j \in \{p+2,\ldots,n\}.$ We see that this is a contradiction to \eqref{eq: b9}, hence
\begin{equation}  \label{eq: b10}
\alpha_i = 0
\end{equation}
for $i \in \{1,\ldots,p+1\}.$

Summarizing these computations, we obtain the following characterization.

\begin{proposition}
Let $\mathcal{X} = \sum R_i \frac{\partial}{\partial x_i}$ be a degree two vector field on $S_{p,q}$ where $R_i$s are given by \eqref{eq: b6}. Let $i,j,k \in \{ 1,\ldots,p+1    \}$ and $l,m,n \in \{ p+2,\ldots,n     \},$ all of them pairwise distinct. Then the co factor in \eqref{eqn: 33} is given by
$$ K = (\alpha_{p+2} x_{p+2} + \cdots + \alpha_n x_n),    $$
this means $\alpha_1 = \alpha_2 = \cdots = \alpha_{p+1} = 0.$

   Also, the coefficients of degree two terms of $R_i$s satisfy
 \begin{equation*}
    \begin{aligned}
    \beta_{ijk} + \beta_{jik} + \beta_{kij} = 0, \quad \mbox{and} \quad \beta_{lmn} + \beta_{mln} + \beta_{nlm} = 0, \\
    \beta_{ijl} + \beta_{jil} = 0, \quad \mbox{and} \quad \beta_{ljm} + \beta_{mjl} = 0, \\
    4(\beta_{ijj} + \beta_{jij}) = \alpha_i, \quad \mbox{and} \quad 2(\beta_{lmm} + \beta_{mlm}) = \alpha_l,  \\
    4 \beta_{iii} = \alpha_i, \quad 2 \beta_{lii} = -a^2 \alpha_l, \quad \mbox{and}  \quad 2 \beta_{lll} = \alpha_l.
    \end{aligned} 
    \end{equation*}
For the coefficients of degree one terms of $R_i$s, we have
\begin{equation*}
    \beta_{ij} + \beta_{ji} = 0, \quad \mbox{and} \quad \beta_{lm} + \beta_{ml} = 0.
    \end{equation*}
Moreover, the constant terms of $R_i$s satisfy
\begin{equation*}
 \beta_{i0} = 0, \quad \mbox{and} \quad \beta_{l0} = \frac{(a^4 - 1)}{2} \alpha_l.    
    \end{equation*}

\end{proposition}

\section{Invariant hyperplanes on \texorpdfstring{$S^p \times S^q$}{ product of spheres} }\label{sec:inv_hyperplanes}

 In this section, we are interested in the number of invariant algebraic sets determined by the intersections of $S_{p,q}$ and hypersurfaces determined by polynomials of degree one. We compute an upper bound for the number of invariant meridians and parallels on $S_{p,q}$. We prove a result to compute the number of invariant meridians of a degree one vector field on $S_{p,q}$. We show that a degree 2 vector field on $S_{p,q}$ can have at most $p$ many invariant meridians. Then we discuss some vector fields with the number of invariant meridians close to an upper bound. Let $\mathcal{X} = (R_1, \ldots, R_n)$ be a polynomial vector field and $\deg(R_i)=m_i$. 
Without loss of generality, we may assume that $m_1 \geq m_2 \geq \cdots \geq m_{p+1}$ and that $m_{p+2} \geq m_{p+3}  \geq \cdots \geq m_{n}$.

\begin{proposition}  \label{prop: 31}
Let $S^{n+1}$ be the standard unit $(n+1)$-sphere in $\mathbb{R}^{n+2}$ and $H$ a hyperplane passing through the origin. Then $S^{n+1} \cap H$ is homeomorphic to $S^n$.
\end{proposition}


\begin{proposition} \label{prop: 32}
On $S_{p,q}$, with $p,q \geq 2$, the  meridians and parallels are path connected.
\end{proposition}

\begin{proof}
Let $\sum_{i=1}^{p+1} a_ix_i=0$ be the hyperplane $H$, which determines a meridian and $$ \sum_{j=p+2}^n x_j^2 = \alpha \in [0,1].   $$
Then \eqref{eq:30} gives the following pair of spheres $$ \sum_{i=1}^{p+1} x_i^2 = a^2 \pm  \sqrt{1 - \alpha}        $$
unless $ \alpha =1$.
By Proposition \ref{prop: 31}, the intersection of any of these spheres with the hyperplane $H$ is a $(p-1)$ dimensional sphere for each $\alpha \in [0,1)$. Thus, by a similar argument as in the proof that $S_{p,q}$ is homeomorphic to $S^p \times S^q$, one can show that $S_{p,q} \cap H$ is homeomorphic to $S^{p-1} \times S^q$. Therefore it is path connected since $p,q >1$.

For the case of parallels, let $\sum_{j=p+2}^n b_jx_j =c$ be the hyperplane $H_2$ and $$\sum_{i=1}^{p+1}x_i^2 - a^2 = \beta \in [-1,1].$$
Then \eqref{eq:30} can be written as
$$ \sum_{j=p+2}^n x_j^2 = 1 - \beta^2 \in [0,1].$$

By Proposition \ref{prop: 31}, the intersection of this sphere and $H_2 - c$ is a $q-1$ dimensional sphere unless $\beta = \pm 1$. If $\beta = \pm 1$, then this intersection is a point. Therefore, the intersection $S_{p,q} \cap H_2$ is homeomorphic to $S^p \times S^{q-1}$ if $\beta \neq \pm 1$ since smooth homotopies preserve transversality, see Section 6 in \cite[Chapter 1]{GP74}. This is path connected since $p,q >1$.\end{proof}

\begin{theorem} \label{thm:5}
Suppose that the polynomial vector field $\mathcal{X}$ on $\RR^n$ has finitely many invariant algebraic hypersurfaces. If $p,q \geq 2$, then
\begin{enumerate}
\item the number of invariant meridians of $\mathcal{X}$ in $S_{p,q}$ is at most $$ {p \choose 2}(m_1-1) + \sum_{i=2}^{p+1}m_i +1, $$ 

\item the number of invariant parallels of $\mathcal{X}$ in $S_{p,q}$ is at most $${q \choose 2}(m_{p+2}-1)+ \sum_{j=1}^{q-1} m_{p+j+1} .  $$
\end{enumerate}

\end{theorem} 
 
\begin{proof}
{\bf For (1).} An invariant meridian of $\mathcal{X}$ is given by the intersection of a hyperplane of the form $g:= \sum_{i=1}^{p+1} a_ix_i=0$ with $S_{p,q}$. By Proposition \ref{Prop:1},  if this hyperplane is invariant for the vector field $\mathcal{X}$, then $\sum_{i=1}^{p+1} a_ix_i$ is a factor of the extactic polynomial
\begin{equation*}
\mathcal{E}_{ \{x_1,x_2,\ldots, x_{p+1} \} }(\mathcal{X}) = \mathsf{det} \begin{pmatrix}
x_1 &  x_2 & \cdots &   x_{p+1} \\ R_1 & R_2 & \cdots & R_{p+1} \\ \mathcal{X}(R_1) & \mathcal{X}(R_2) & \cdots & \mathcal{X}(R_{p+1}) \\ \vdots & \vdots & \ddots & \vdots \\ \mathcal{X}^{p-1}(R_1) & \mathcal{X}^{p-1}(R_2) & \cdots & \mathcal{X}^{p-1}(R_{p+1}) \end{pmatrix} .
\end{equation*}

Since we  have chosen  degrees of $R_1,\ldots,R_{p+1}$ in decreasing order, we see that the term 
\begin{equation} \label{eq: 344}
\mathcal{X}^{p-1}(R_{p+1})\cdot \mathcal{X}^{p-2}(R_p)\cdots \mathcal{X}(R_{3}) \cdot R_2 \cdot x_{1} \end{equation}
has the least degree in the polynomial $\mathcal{E}_{ \{x_1,x_2,\cdots, x_{p+1} \} }(\mathcal{X})$.
Now the degree of $\mathcal{X}^{i-1}(R_{i+1})$ is $(i-1)(m_1-1)+m_{i+1}$ for $1 \leq i \leq p$. Therefore,
\begin{equation*}
 \deg (\mathcal{X}^{p-1}(R_{p+1})\cdot \mathcal{X}^{p-2}(R_p)\cdots \mathcal{X}(R_{3}) \cdot R_2 \cdot x_{1}) = {p \choose 2}(m_1-1) + \sum_{i=2}^{p+1} m_i +1.
\end{equation*}
The number of invariant meridians cannot exceed the degree of the polynomial in \eqref{eq: 344}
 since meridians are determined by linear homogeneous polynomials in $\{x_1,\ldots,x_{p+1}\}$. This proves (1).

{\bf For (2).} In this case, the vector space $W$ for $\mathcal{E}_{W }(\mathcal{X})$ is generated by $\{x_{p+2}, \ldots, x_n, 1\}$. Then the result follows using the same argument as in the first case. 
\end{proof}

\subsection{Invariant Meridians for degree one vector fields on \texorpdfstring{$S^p \times S^q$}{ product of spheres } }
In this subsection, we give a tight upper bound on the number of invariant meridians for degree one vector fields on $S_{p,q}$. For these vector fields, $\mathcal{X} =  \sum_i R_i \frac{\partial}{\partial x_i} $ where $\deg(R_i) \leq 1$ and at least one of them has degree one.
Let $\mathcal{X}$ be invariant on $\sum_{i=1}^{p+1} a_ix_i = 0$. So,
\begin{equation} \label{eqn: a1}
\sum_{i=1}^{p+1} a_iR_i = K' (\sum_{i=1}^{p+1} a_ix_i ),
\end{equation}
for some $K' \in \mathbb{R}[x_1,\ldots,x_{p+1},x_{p+2},\ldots,x_n].$  We get that $K' $ of \eqref{eqn: a1} is a constant since the left hand side has degree one and the term in the brackets on the right  has degree one.

Further substituting for $R_i$ from \eqref{eq: a21} in \eqref{eqn: a1}, we get
\begin{equation*}
\sum_{i=1}^{p+1}a_i \sum_{j=1}^{p+1}c_{ij}x_j = K' \left( \sum_{s=1}^{p+1}a_sx_s \right).
\end{equation*} 
This can be written as 
\begin{equation*} \label{eq: 32}
\sum_{i=1}^{p+1} \sum_{j=1}^{p+1}a_ic_{ij}x_j = K' \left( \sum_{s=1}^{p+1}a_sx_s \right).
\end{equation*}  
Reindexing and equating coefficients of $x_j$, we have
\begin{equation} \label{eq: 34}
\sum_{i=1}^{p+1}a_ic_{ij} = K'a_j.
\end{equation} 
Thus, $(a_1,\ldots,a_{p+1})$ is an eigenvector for the real matrix $A:=(c_{ij})_{1\leq i,j \leq p+1}$ with the eigenvalue $K'$. Since the matrix $(c_{ij})$ is skew symmetric by Proposition \ref{prop: bx32}, 0 is the only possible real eigenvalue (see Exercise 7.(g) in Section 6.6 of \cite{FSIS97}).
As a consequence, we get the following which helps to compute the number of invariant meridians of a degree one vector field on $S_{p,q}$.

\begin{theorem} 
A degree one vector field defined on $S_{p,q}$ can have at most as many invariant meridians as there are linearly independent real eigenvectors of the matrix $A = (c_{ij})_{1 \leq i,j \leq p+1}$ formed by the coefficients of the polynomials in the vector field $\mathcal{X}$. 
 \end{theorem}

\begin{remark} \label{re:30}
The real skew symmetric matrix $A=(c_{ij})$ is normal, that is, it commutes with its adjoint (which, in this case, is the transpose). Hence by the Spectral Theorem (Theorem 2.5.4 in page 101 of \cite{HorJon}) $A$ is diagonalizable. So if all eigenvalues of $A$ are zero, $A$ is the zero matrix. Hence $A$ cannot have all eigenvalues zero, since $\mathcal{X}$ is a non-zero vector field. Thus, we see that a degree one vector field on $S_{p,q}$ can have at most $(p-1)$ invariant meridians. In fact, we can readily construct a vector field with $(p-1)$ invariant meridians starting with a $(p+1) \times (p+1)$ skew-symmetric matrix with $(p-1)$ eigenvalues zero. \end{remark} 
 
\begin{example}
Let $p=2$ in $S_{p,q}$. Then $\mathcal{X}=(-x_2, x_1-x_3, x_2)$ is a vector field on $S_{p,q}$.  So
\begin{equation*}
A= \begin{pmatrix}
0 &  -1 & 0 \\ 1 &  0 & -1 \\ 0 & 1 & 0\end{pmatrix}.
\end{equation*}
Note that the eigenvalues of $A$ are $\pm \sqrt{-2}, 0$. In this case, the unique eigenvector is $(1,0,1)$, and hence the corresponding invariant meridian is given by $\{ x_1+x_3=0 \}$. Further, we see that this matches with the bound in Remark \ref{re:30} which is 1 in this case.
\end{example}

\subsection{General Case}

Now we look at the case of vector fields of degree greater than or equal to four.

\begin{theorem}   \label{thm:33}
Assume $m \geq 4$. We have the following.
\begin{enumerate}

\item  There exists a vector field of degree $m$ with $3m-10$ invariant meridians counted with multiplicity when $p=2$.

\item  There exists a vector field of degree $m$ with $6m-21$ invariant meridians when $p=3$.

\end{enumerate}
\end{theorem}

\begin{proof}
First we construct $R_j$s, when $j \in \{ p+2,\ldots,n   \},$ for the invariant vector field $\mathcal{X}$.
Let $R_j$s be given by
\begin{equation*}
\begin{split}
& R_{p+2} = x_1(x_{p+2}^2-1)H, \quad R_{p+3} = x_1x_{p+2}x_{p+3}H, \quad R_{p+4}= x_1x_{p+2}x_{p+4}H,  \\ & R_{p+5} = x_1x_{p+2}x_{p+5}H,  \quad R_{p+6} = x_1x_{p+2}x_{p+6}H,\quad \ldots, \textsf{and}, \quad R_n = x_1x_{p+2}x_{n}H \end{split} \end{equation*}
where $H \in \mathbb{R}[x_1,\ldots,x_{p+1},x_{p+2},\ldots,x_n]$. 

{\bf For (1)}: Let $$D := (\sum_{i=1}^{3} a_ix_i)^{m-3},$$ and
 \begin{equation*}
R_1 = \frac{1}{2}x_{4}(x_1^2-a^2) D, \quad R_2 = \frac{1}{2}x_{4}x_1x_2 D \quad \textsf{and} \quad
R_3 = \frac{1}{2}x_{4}x_1x_3 D.\end{equation*}

Observe that the vector field $\mathcal{X}$ determined by the above choices for $R_1,R_2,R_3$ and $R_{4},\ldots,R_n$  is invariant on $S_{p,q}$ if we let $H = D$. To be precise, for these $R_i$s , $\mathcal{X}$ satisfies equation \eqref{eqn: 33} with $K = 2x_1x_{4}$.
We shall prove that the vector field $\mathcal{X}$, defined by these $R_i$s  have $\sum_{i=1}^{3} a_ix_i = 0$ as an invariant meridian with multiplicity $3m-10$ for $\mathcal{X}$.
The polynomial $\sum_{i=1}^{3}a_ix_i$ belongs to the vector space $W := \textsf{ span} \{x_1,x_2,x_3 \}$. We note that if $\sum_{i=1}^{3}a_ix_i=0$ gives an invariant meridian for this vector field $\mathcal{X}$, then $\sum_{i=1}^{3}a_ix_i$ is a factor of the corresponding extactic polynomial
$$  \mathcal{E}_W(\mathcal{X}) = \mathsf{det} \begin{pmatrix}

x_1 &  x_2 &   x_{3} \\ \mathcal{X}(x_1) & \mathcal{X}(x_2) &  \mathcal{X}(x_{3})  \\ \mathcal{X}^{2}(x_1) & \mathcal{X}^{2}(x_2) &  \mathcal{X}^{2}(x_{3}) \end{pmatrix} .  $$
 In our case, this is the following 
\begin{gather} \label{eq: 320}
\mathsf{det} \begin{pmatrix}
x_1 &  x_2 &   x_{3} \\ \frac{1}{2}x_{4}(x_1^2-a^2) D & \frac{1}{2}x_{4}x_1x_2 D  & \frac{1}{2}x_{4}x_1x_3 D  \\ g_1D^2 + h_1E & g_2D^2 +h_2E &  g_3D^2 + h_3E \end{pmatrix} 
\end{gather}
where $$E := \mathcal{X}(D) = (m-3)(\sum_{i=1}^{3} a_ix_i)^{m-4}(\sum_{i=1}^{3} a_ix_{\sigma(i)})\cdot D  $$
$$= (m-3)(\sum_{i=1}^{3} a_ix_i)^{2m-7}(\sum_{i=1}^{3} a_ix_{\sigma(i)})    $$
for some  permutation $\sigma$ on $\{1,2,3 \}$, and $g_i,h_i \in \mathbb{R}[x_1,\ldots,x_{3},x_{4},\ldots,x_n]$ for $i \in \{1,2,3 \}$.
The third row consists of the terms like $$ \mathcal{X}(\frac{1}{2}x_4(x_1^2-a^2) D) = \frac{1}{2}(x_4^2-1) (x_1^2-a^2)D^2 + \frac{1}{2}x_1x_4^2(x_1^2-a^2)D^2 + \frac{1}{2}x_4(x_1^2-a^2)E  $$
which we have written as $g_1D^2+h_1E$. 
Since $D^2 = (\sum_{i=1}^{3} a_ix_i)^{2m-6}$, in the third row, one sees that $(\sum_{i=1}^{p+1} a_ix_i)^{2m-7}$ is a common factor of each term in the third row of the matrix in \eqref{eq: 320}. Also in the second row, $D$ is a common factor, hence
$$ \mathcal{E}_W(\mathcal{X}) = D \cdot E \cdot h'(x_1,x_2,x_3,x_4) = (\sum_{i=1}^{3} a_ix_i)^{3m-10} h(x_1,x_2,x_3,x_4)     $$
for some polynomials $h'$ and $h$ in $\mathbb{R}[x_1,x_2,x_3,x_4,\ldots,x_n]$. Thus (1) is proved, since $(\sum_{i=1}^{3} a_ix_i)$ divides $\mathcal{E}_W(\mathcal{X})$ with multiplicity $3m-10$.

{\bf For (2)}: Let $R := (\sum_{i=1}^{4} a_ix_i)^{m-3}$ and 
\begin{gather*}
R_1 = \frac{1}{2}x_5(x_1^2-a^2) R,~ R_2 = \frac{1}{2}x_5x_1x_2R,~ R_3 = \frac{1}{2}x_5x_1x_3R , ~ \textsf{and},~~ R_4 = \frac{1}{2}x_5x_1x_4R. \end{gather*}

Observe that $(\sum_{i=1}^4 a_ix_i)$ is a polynomial in $W = \textsf{span} \{x_1,x_2,x_3,x_4 \}$ and the corresponding extactic polynomial is 
  $$  \mathcal{E}_W(\mathcal{X}) = \mathsf{det} \begin{pmatrix}
x_1 &  x_2 &   x_{3} & x_4 \\ \mathcal{X}(x_1) & \mathcal{X}(x_2) &  \mathcal{X}(x_{3}) & \mathcal{X}(x_4) \\ \mathcal{X}^{2}(x_1) & \mathcal{X}^{2}(x_2) &  \mathcal{X}^{2}(x_{3}) & \mathcal{X}^{2}(x_4) \\ \mathcal{X}^{3}(x_1) & \mathcal{X}^{3}(x_2) & \mathcal{X}^{3}(x_3) & \mathcal{X}^{3}(x_4)\end{pmatrix} .  $$

In order to compute $\mathcal{E}_W(\mathcal{X})$ in this case, let us first determine
$$ T := \mathcal{X}(R) = (m-3) (\sum_{i=1}^{4} a_ix_i)^{2m-7}(\sum_{i=1}^{4} a_ix_{\sigma(i)}),$$
and
\begin{multline*} 
U := \mathcal{X}(T) = (m-3)(2m-7)(\sum_{i=1}^{4} a_ix_i)^{3m-11}(\sum_{i=1}^{4} a_ix_{\sigma(i)})^2 \\ + 2(m-3)(\sum_{i=1}^{4} a_ix_i)^{3m-10}(\sum_{i=1}^{4} a_ix_{\sigma'(i)})        \end{multline*}
where $\sigma$ and $\sigma'$ are some permutations  on the set $\{1,2,3,4 \}$.
Then $\mathcal{E}_W(\mathcal{X})$  becomes
\begin{gather} \label{eq: 39}
\mathcal{E}_W(\mathcal{X}) = \mathsf{det} \begin{pmatrix}
x_1 &  x_2 &   x_{3} & x_4  \\ \frac{1}{2}x_5(x_1^2-a^2) R & \frac{1}{2}x_5x_1x_2 R  & \frac{1}{2}x_5x_1x_3 R & \frac{1}{2}x_5x_1x_4 R \\   f_{31}R^2+g_{31}T & f_{32}R^2+g_{32}T & f_{33}R^2+g_{33}T &    f_{34}R^2+g_{34}T \\ f_{41}R^3+g_{41}RT+h_{41}U &  f_{42}R^3+g_{42}RT+h_{42}U  & f_{43}R^3+g_{43}RT+h_{43}U & f_{44}R^3+g_{44}RT+h_{44}U \end{pmatrix}   \end{gather}
where the $f_{ij},g_{lk},h_{st} \in \mathbb{R}[x_1,\ldots,x_{4},x_5,\ldots,x_n]$ for $i,l,s \in \{3,4\}$ and $j,k,t \in \{1,2,3,4\}$.

Notice that $R$ is a factor of each term of the second row, $(\sum_{i=1}^{4} a_ix_i)^{2m-7}$ is a factor of each term of the third row, and $(\sum_{i=1}^{4} a_ix_i)^{3m-11}$ is a factor of  each term of the fourth row of the matrix in \eqref{eq: 39}. Therefore, the extactic polynomial can be written as 
$$  \mathcal{E}_W(\mathcal{X}) = (\sum_{i=1}^{4} a_ix_i)^{6m-21}\cdot h(x_1,x_2,x_3,x_4,x_5)   $$
for some $h \in \mathbb{R}[x_1,\ldots,x_{4},x_5] \subset \mathbb{R}[x_1,\ldots,x_{4},x_5,\ldots,x_n].$
This proves the claim (2), since $(\sum_{i=1}^{4} a_ix_i)$ divides $\mathcal{E}_W(\mathcal{X})$ with multiplicity $6m-21$.
\end{proof}
\begin{remark}
For $p=2$ and $p=3$, an upper bound given by Theorem \ref{thm:5}  is $3m$ and  $6m-2$ respectively. 
\end{remark}

\begin{example} 
Suppose that the vector field $\mathcal{X}$ is given by
  \begin{multline}  \label{eq: 3b18}
 R_1= \frac{1}{2}x_{p+2}(x_1^2-a^2) G, ~~~~ R_2 = \frac{1}{2}x_{p+2}x_1x_2G, \ldots , R_{p+1}= \frac{1}{2}x_{p+2}x_1x_{p+1}G, \\ R_{p+2}= x_1(x_{p+2}^2-1)G, ~~~~ R_{p+3} = x_1x_{p+2}x_{p+3}G, ~~~~ R_{p+4} = x_1x_{p+2}x_{p+4}G, \\  R_{p+5} = x_1x_{p+2}x_{p+5}G,  ~~~~ R_{p+6} = x_1x_{p+2}x_{p+6}G, ~~~~\ldots, ~~~~ R_n  =x_1x_{p+2}x_nG, \end{multline}
  where $G \in \mathbb{R}[x_1,\ldots,x_{p+1},x_{p+2}, \ldots , x_n]$ is a polynomial  of degree  $(\deg \mathcal{X}-3)$.
 Consider the hyperplane given by $x_i = c$ where $c$ is a constant and $-1 < c < 1$. We want to look at the invariant algebraic sets formed by the intersection of these hyperplanes with $S_{p,q}$. The number of connected components of $ \{ x_i - c \} \cap S_{p,q}   $ is one, since $ \{ x_i - c \} \cap S_{p,q}   $ is homeomorphic to $S^{p-1} \times S^q$ for $p \geq 2$, $q \geq 2$.
In this case, the extactic polynomial is
\begin{equation*}
\mathcal{E}_{\{1, x_i \}}(\mathcal{X}) = \mathsf{det} \begin{pmatrix}
1 &  x_i \\ 0 &  R_i\end{pmatrix}  = R_i.
\end{equation*}
We see that the maximum possible number of these invariant hyperplanes is $m_i (= \deg R_i)$. In particular, letting $G = \prod_{j=1}^{m_i-3} (x_i - c_j)$ in the vector field given by \eqref{eq: 3b18}, and if we regard $x_i-0=x_i$ also as one of the invariant hyperplanes ($x_1-a=0$ and $x_1+a=0$ if $i=1$), we  see that we have  $m_i-1$ invariant algebraic sets of the form under consideration for this choice of the vector field.

Similarly, one can do the computation for $x_j=c$, when $j \in \{ p+2,\ldots,n   \},$ and an almost tight bound can be obtained.
\end{example}

\section{Invariant Algebraic sets on \texorpdfstring{$S^1 \times S^q$ }{product of circle and sphere}}\label{sec:inv_planes_s1sq}
In this section, we give a tight upper bound for the number of invariant hyperplanes of certain types for the vector fields on $S^1 \times S^q$. In section 3, we showed that $S_{1,q} \cong S^1 \times S^q $.

 If the meridians on $S_{1,q}$ are  invariant algebraic sets, then $a_1x_1+a_2x_2$ divides the extactic polynomial 
\begin{equation} \label{eq: 43}
\mathcal{E}_{ \{x_1,x_2 \} }(\mathcal{X})= \mathsf{det} \begin{pmatrix} x_1 & x_2 \\ \mathcal{X}(x_1) & \mathcal{X}(x_2) \end{pmatrix} = \mathsf{det} \begin{pmatrix}  x_1 & x_2 \\ R_1 & R_2 \end{pmatrix} = x_1R_2 - x_2R_1 .
\end{equation}
 If the parallels on $S_{1,q}$ are invariant algebraic sets, then $\sum_{j=3}^n b_jx_j-c$ divides the extactic polynomial 
\begin{equation*} 
\mathcal{E}_{\{1,x_{3},\dots,x_n  \}} = \mathsf{det} \begin{pmatrix} 1 & x_{3} & \cdots & x_n \\ 0 & \mathcal{X}(x_{3}) & \cdots & \mathcal{X}(x_n) \\ \vdots & \vdots & \ddots & \vdots \\ 0 & \mathcal{X}^{q-1}(x_{3}) & \cdots & \mathcal{X}^{q-1}(x_n) \end{pmatrix} .
\end{equation*}

\begin{proposition} \label{prop:40}
The numbers of connected components of the intersections $ \{ a_1x_1 + a_2x_2 = 0 \} \cap S_{1,q} $ and $ \{ \sum_{j=3}^n b_jx_j = c \} \cap S_{1,q} $ are two and one respectively.
\end{proposition}

\begin{proof}
For the case $ \{ a_1x_1 + a_2x_2 = 0 \} \cap S_{1,q}$,  put $x_2 = -a_1x_1/a_2$, if $a_2 \neq 0$, otherwise $x_1 = -a_2x_2/a_1$. Then \eqref{eq:30} may have the following form.
$$ (x_1^2(1+ \frac{a_1^2}{a_2^2}) - a^2)^2 + \sum_{j=3}^n x_j^2 =1. $$
That is
$$ x_1^2(1+ \frac{a_1^2}{a_2^2}) = a^2 \pm \sqrt{1 - \sum_{j=3}^n x_j^2} . $$
For a fixed value of $x_1$, we have that $ \sum_{j=3}^n x_j^2$ is a constant, which means we get a copy of $S^{q-1}$. Fixing the sign of $x_1$, we see that when the expression $\sqrt{1 - \sum_{j=3}^n x_j^2 }$ is zero, the copies of $S^{q-1}$ corresponding to the plus and minus signs inside the radical on the right hand side coincide. This implies that they belong to the same component. The two different possible signs for $x_1$ give two distinct components since $x_1$ cannot be equal to zero since $a>1$.
 The argument for the case of $(\sum_{j=3}^n b_jx_j =c) \cap S_{1,q}$ is similar to the proof of the second part of Proposition \ref{prop: 32}.
\end{proof}

\begin{theorem}   \label{thm:2}
Let $\mathcal{X}$  be a polynomial vector field of degree $m$  on $S_{1,q}$ such that there are only finitely many invariant algebraic sets. Then we have the following.
\begin{enumerate}

\item  There can be at most $2(m-1)$  invariant meridians.
 
\item There exists a polynomial vector field on $S_{1,q}$ with exactly $2(m-1)$ invariant meridians if $(m-1) \geq q$.
\end{enumerate}
\end{theorem}
 
\begin{proof}
{\bf  For (1)}: Let $x_1=r\cos \theta$, and $x_2=r\sin\theta $ where $r \geq 0$ and $\theta \in [0, 2 \pi]$.  With these choices of coordinates, the polynomial vector field becomes
\begin{equation*}
 \begin{split}
 \mathcal{X} & =  \frac{1}{r}(R_1(r \cos \theta, r \sin \theta, x_{3}, \cdots, x_n)r \cos \theta \\ & + R_2(r \cos \theta, r \sin \theta, x_3, \cdots, x_n) r \sin \theta ) \frac{\partial}{\partial r} \\
 & - \frac{1}{r^2}(R_1(r \cos \theta, r \sin \theta, x_3, \cdots, x_n)r \sin \theta \\ & - R_2(r \cos \theta, r \sin \theta, x_3, \cdots, x_n)r \cos \theta)\frac{\partial}{\partial \theta} 
 + \sum_{i=3}^n R_i \frac{\partial}{\partial x_i }.
 \end{split}
\end{equation*}  
This implies that 
\begin{equation} \label{eq: 46}
 x_1R_2 - x_2R_1 = r^2 \dot{\theta} = (x_1^2+x_2^2)\dot{\theta}.\end{equation}
Observe that the maximum degree of the left hand side of \eqref{eq: 46} is $m+1$ and  $x_1^2+x_2^2$ 
is irreducible over $\mathbb{R}$. By Proposition \ref{prop:40}, we know that the intersection $\{a_1x_1+a_2x_2 \} \cap S_{1,q}$ has two connected components. Hence there can be at most $2(m-1)$  invariant meridians on $S_{1,q}$. This proves the claim (1).

 {\bf For (2)}: We consider the vector field $\mathcal{X}$ on $S_{1,q}$ given by 
$$ R_1 = x_1x_3 \cdots x_n -x_2G, \quad R_2 =  x_2x_3 \cdots x_n +x_1G, $$ and
$$ R_j = \frac{2}{q} \left( -a^2(x_1^2+x_2^2) + \sum_{s=3}^n x_s^2 -1 \right)x_3 \cdots x_{j-1}x_{j+1} \cdots x_n $$
for $j=3, \ldots, n$. Taking $$ G = \prod_{i=1}^{m-1} (a_ix_1+b_ix_2), $$
one can see that $G$ is a factor of  $\mathcal{E}_{x_1,x_2}(\mathcal{X})$ of \eqref{eq: 43}.  We know that the intersection $\{a_1x_1+a_2x_2 \} \cap S_{1,q}$ has two connected components. Therefore, the number of invariant meridians for this vector field is $2(m_1-1)$. 
\end{proof}

\begin{remark}
An upper bound for the number of invariant parallels for the vector field $\mathcal{X}$ on $S_{1,q}$ can be given by a similar calculation as in Theorem \ref{thm:5} (2). This upper bound  is ${q \choose 2}(m_{p+2}-1)+ \sum_{j=1}^{q-1} m_{p+j+1}$. 
\end{remark}

\begin{example}
Consider the vector field $\mathcal{X} = (R_1,\ldots,R_n)$ on $S_{1,q}$ given by
\begin{equation*}
 R_1  =  \frac{x_1}{4}(\sum_{s=3}^n x_s), ~ R_2 =  \frac{x_2}{4}(\sum_{s=3}^n x_s), ~ R_i = \frac{x_i}{4}(\sum_{s=3}^n x_s) - \frac{a^2}{2}(x_1^2 + x_2^2) + \frac{(a^4 -1)}{2},
\end{equation*}
for $i=3, \ldots, n$.
One can check that this vector field satisfies \eqref{eq:30} with cofactor $K = (x_3 + x_4 + \cdots + x_n).$  Notice that all meridians are invariant for this vector field, that is
$$ \mathcal{X}(a_1x_1 +  a_2x_2) = K (a_1x_1 + a_2x_2)        $$
with $K = \frac{1}{4}(\sum_{s=3}^n x_s).$ Hence by Proposition \ref{thm:1}, $\mathcal{X}$ has a rational first integral.
\end{example}


\section{Invariant Algebraic sets on \texorpdfstring{$S^p \times S^1$}{ product of sphere and circle}}\label{sec:inv_planes_sps1}
In this section, we study  invariant meridians and parallels of the vector fields defined on $S^p \times S^1$ for $p \geq 2$.

If a meridian in $S_{p,1}$ is an invariant algebraic set, then $\sum_{i=1}^{p+1} a_ix_i$ divides the extactic polynomial
\begin{equation}  \label{eq:52}
\mathcal{E}_{ \{x_1,\dots,x_{p+1}   \} }(\mathcal{X}) =\mathsf{det} \begin{pmatrix} x_1 & \cdots & x_{p+1} \\ \mathcal{X}(x_1) & \cdots & \mathcal{X}(x_{p+1}) \\ \vdots & \ddots & \vdots \\ \mathcal{X}^{p}(x_1) & \cdots & \mathcal{X}^{p}(x_{p+1}) \end{pmatrix}.
\end{equation}

If a parallel in $S_{p,1}$ is an invariant algebraic set, then $ x_{p+2}-c $ divides the extactic polynomial
\begin{equation} \label{eq:62}
\mathcal{E}_{\{1,x_{p+2} \} }(\mathcal{X})= \mathsf{det} \begin{pmatrix} 1 & x_{p+2} \\ \mathcal{X}(1) & \mathcal{X}(x_{p+2}) \end{pmatrix} = \mathsf{det} \begin{pmatrix}  1 & x_{p+2} \\ 0 & R_{p+2} \end{pmatrix} = R_{p+2}.
\end{equation} 

\begin{proposition} \label{prop: 61}
The numbers of connected components of the intersections $ \{ \sum_{i=1}^{p+1} a_ix_i = 0 \} \cap S_{p,1} $ and $ \{ x_{p+2} = c  \} \cap S_{p,1}  $ are one and two respectively for $|c|<1$.
\end{proposition}

\begin{proof}
The argument for the case of $ \{ \sum_{i=1}^{p+1} a_ix_i = 0 \} \cap S_{p,1} $ is similar to the proof of the first part of Proposition \ref{prop: 32}.

 For the case of $ \{ x_{p+2} = c  \} \cap S_{p,1}  $,  \eqref{eq:30} becomes
$$ (\sum_{i=1}^{p+1}x_i^2 - a^2)^2 + c^2 = 1 ~~\mbox{or},~~ \sum_{i=1}^{p+1} x_i^2 = a^2 \pm \sqrt{1-c^2}.  $$
Since $c$ is a fixed constant, this gives two different concentric spheres. Therefore $\{ x_{p+2} =c  \} \cap S_{p,1}$ has two connected components, unless $|c| \geq 1$.
\end{proof} 

We note that by the arguments in the proof of Theorem \ref{thm:5} (1) we know  that the maximum number of invariant meridians is $$ {p \choose 2}(m_1-1) + \sum_{i=2}^{p+1}m_i +1.  $$

\begin{proposition}
Let $\mathcal{X}$ be a vector field on $S_{p,1}$ such that $\mathcal{X}$ has only finitely many invariant algebraic sets. Then the maximum number of invariant parallels for $\mathcal{X}$ is $(m-1)$ where $m = \deg \mathcal{X}$. Moreover, there is a vector field on $S_{p,1}$ with exactly $(m-1)$ invariant parallels in $S_{p,1}$.
\end{proposition}

\begin{proof}
The vector field $\mathcal{X}$ is invariant on $S_{p,1}$. So, by definition, we have 
\begin{gather}  \label{eq: ax51}
4\left( \sum_{i=1}^{p+1}x_i^2 - a^2 \right) \left(\sum_{i=1}^{p+1}x_iR_i \right)+2x_{p+2}R_{p+2} ~ = ~ K(x_1,\ldots,x_{p+1},x_{p+2}) \left(\left(\sum_{i=1}^{p+1} x_i^2 - a^2 \right)^2 +x_{p+2}^2 -1 \right). \end{gather}
If $x_{p+2}-a_i=0$ gives a parallel, then $x_{p+2}-a_i$ is a factor of $R_{p+2}$ by Proposition \ref{Prop:1} and \eqref{eq:62}.  
Then $R_{p+2}$ must be of the form
$$ R_{p+2} = \prod_{i=1}^{\ell} (x_{p+2} - a_i) \cdot h(x_1,\ldots,x_{p+1}), $$
where $0 \leq \ell \leq \deg R_{p+2} \leq m$ and $h$ has no factor of the form $x_{p+2} - a$.

Suppose that $\ell = m$. Then $h$ is a constant, since $\deg R_{p+2} \leq m.$ Now putting $x_1 = x_2 = \cdots = x_{p+1} = 0$ in \eqref{eq: ax51}, we have the following. 
$$ 2x_{p+2} \prod_{i=1}^{\ell} (x_{p+2} - a_i) \cdot h = K(0,\ldots,0,x_{p+2})(x_{p+2}^2 + a^4 -1),$$
where $a>1$. 
We see that the left hand side is non-zero. So, $K(0,\ldots,0,x_{p+2})$ is also non-zero. The polynomial $(x_{p+2}^2 + a^4 -1)$ is irreducible 
 over $\mathbb{R},$ since $a >1$. But all factors in the left hand side have degree one. So, we arrive at a contradiction. Thus, we have that $\ell \leq (m-1)$ and the proof is complete for the first part.

Now, consider the vector field $\mathcal{X}$ given by
\begin{equation} \label{eq: 541}
\begin{split}
& R_i = x_{p+2} \frac{\left(\sum_{i=1}^{p+1}x_i^2 - a^2 \right)}{4}H  \\
& R_{p+2} = \frac{\left( \sum_{i=1}^{p+1}x_i \right) (x_{p+2}^2 - 1)}{2} H,
\end{split}
\end{equation}
where $R_i \in \mathbb{R}[x_1,x_2,\ldots, x_{p+1},x_{p+2}]$ for $i = 1,\ldots,p+1$.
Observe that $\mathcal{X}$ is a vector field on $S_{p,1}$. Let
$ H := \prod_{i=1}^{m-3} (x_{p+2}-c_i)$
for $c_i \in \mathbb{R}$ with $|c_i| < 1$. Then  from \eqref{eq:62}, we see that $(x_{p+2}^2 - 1)H$ divides the extactic polynomial $\mathcal{E}_{\{1,x_{p+2} \}}(\mathcal{X})$. Therefore, the vector field given by \eqref{eq: 541} has $(m-1)$ invariant parallels in $S_{p,1}$. 
\end{proof} 

 We remark that there exists $2(m-3) + 2 = 2(m-2)$ connected components in this case, since each parallel has two connected components if $-1 < a_i <1$ by Proposition \ref{prop: 61}.

We note that when $p=q$, then $S_{p,1}$ and $S_{1,q}$ are homeomorphic. However, their equations say that they are different algebraic subsets of $\mathbb{R}^{p+2}$.

\begin{example}
Consider the polynomial vector field  determined by 
\begin{equation*}
 R_i  =  \frac{1}{4}x_ix_{p+2} \quad \mbox{and}  \quad R_{p+2} =  \frac{x_{p+2}^2}{2} - \sum_i^{p+1} \frac{a^2}{2} x_i^2 + \frac{(a^4-1)}{2}
\end{equation*}
for $i=1, \ldots, p+1$.
One can check that this vector field satisfies \eqref{eq: ax51}, with the cofactor $x_{p+2}$. Hence $\mathcal{X}=(R_1, \ldots, R_{p+1})$ is a vector field on $S_{p, 1}.$ Notice that all meridians on $S_{p,1}$ are invariant for this vector field, that is
$$\mathcal{X}(a_1x_1+ \cdots + a_{p+1}x_{p+1}) = K (a_1x_1 + \cdots + a_{p+1}x_{p+1}),$$
with $K = \frac{x_{p+2}}{4}.$ Hence by Proposition \ref{thm:1}, $\mathcal{X}$ has a rational first integral.
\end{example}

\noindent {\bf Acknowledgments.}
The authors thank the anonymous referee for several helpful comments and suggestions. The first author was supported by a Senior Research Fellowship of the University Grants Commission of India for the duration of this work. The second author thanks `ICSR office IIT Madras' for SEED research grant.

\bibliography{Ref_InvariantAlgebraicSets}{}
\bibliographystyle{plain}

\end{document}